\documentclass[12pt,a4paper]{article}

\usepackage{amsmath,amssymb,color} 
\usepackage[T1]{fontenc}
\usepackage{lmodern}
\usepackage{latexsym}
\usepackage{amsfonts}
\usepackage{amssymb}
\usepackage{mathtools}
\usepackage{amsthm}
\usepackage{multirow}
\usepackage{wrapfig}
\usepackage{float}

\newtheorem{theorem}{Theorem} 
\newtheorem{lemma}[theorem]{Lemma}

\newtheorem{claim}{Claim}

\newcommand{\HH}{\mathcal{H}}
\newcommand{\FF}{\mathcal{F}}
\newcommand{\DD}{\mathcal{D}}
\newcommand{\CC}{\mathcal{C}}

\newcommand{\rank}{\operatorname{rank}}

\newcommand{\integerset}{\mathbb{Z}}

\date{\empty}

\begin{document} 

\title{Berge $k$-Factors of Regular Hypergraphs}

\author{Mikio Kano$^1$\thanks{ mikio.kano.math@vc.ibaraki.ac.jp}, 
Shun-ichi Maezawa$^2$\thanks{maezawa.mw@gmail.com}, 
Akira Saito$^2$\thanks{saitou.akira@nihon-u.ac.jp}, \\
Kiyoshi Yoshimoto$^3$\thanks{kiyoshi.yoshimoto@gmail.com} \\
$^1$  Ibaraki University, Hitachi, Ibaraki, Japan. \\
$^2$ Nihon University, Setagaya-ku, Tokyo, Japan. \\
$^3$ Nihon University, Chiyoda-ku, Tokyo, Japan.
}

%%%%%%%%%%%%%%%%%%%%%%%%%%%%%%%%%%%%%%%%%%%%%%%%%%%%%%%%%%%%%%%%%%%%%%%%%%%%%%%%%%%%%%%%%
\maketitle

\begin{abstract}
A Berge $k$-factor in a hypergraph is a generalization
of a $k$-factor in a graph.
In this paper,
we study the problem of determining the values $k$
such that every $\lambda$-edge-connected $r$-regular hypergraph
$\HH$ with $k|V(\HH)|$ even has a Berge $k$-factor.
While this problem is completely solved for ordinary graphs,
we report that there arises a new upper bound to $k$ based on the rank of $\HH$
for hypergraphs and that it is stronger than the classical upper bound
based on the edge-connectivity in most cases. 
\end{abstract}

\section{Introduction}
In this paper,
we extend some results on factors in graphs to those in hypergraphs.
When we consider factors, 
we often deal with graphs which may have multiple edges but no loops.
 Thus, throughout this paper, a graph means a multigraph, whereas a graph with neither loops nor multiple edges is called a {\em simple graph}.
Following this convention,
we allow hypergraphs to have multiple edges as well.
Moreover,  since edges consisting of one vertex play no role in this paper, 
we assume that
every edge consists of at least two vertices.
To make these conventions clear,
we begin with the definition of a hypergraph.

A hypergraph $\HH$ is a pair $(V(\HH), E(\HH))$,
where $V(\HH)$ is a finite set and $E(\HH)$ is a finite multiset.
An element of $V(\HH)$ is called a vertex, and an element of
$E(\HH)$ is called an edge, which consists of at least two vertices of  $V(\HH)$, 
where multiple edges are allowed.
 We adopt set-theoretic terminology for $E(\HH)$, for example, 
 the size of an edge $e$ is denoted by $|e|$.

The maximum size of an edge in a hypergraph $\HH$
is called the {\em rank} of $\HH$ and denoted by $\rank(\HH)\colon$
$\rank(\HH)=\max\{|e|\colon e\in E(\HH)\}$.
If $|e|=t$ for every $e\in E(\HH)$,
we say that $\HH$ is {\em $t$-uniform}.
Thus,
an ordinary graph is a $2$-uniform hypergraph and also a hypergraph with rank 2.
For a vertex $v$ in $\HH$,
the number of edges that contain $v$ is called the degree
of $v$ in $\HH$ and denoted by $\deg_{\HH}(v)$.
If all vertices of $\HH$ have the same degree $r$, 
we say that $\HH$ is {\em $r$-regular}.
A hypergraph $\HH'$ with $V(\HH')\subseteq V(\HH)$
and $E(\HH')\subseteq E(\HH)$ is
called a subhypergraph of $\HH$.
In particular,
if a subhypergraph $\HH'$ of $\HH$ satisfies $V(\HH')=V(\HH)$, then
we say that $\HH'$ is a spanning subhypergraph of $\HH$.
For $X\subset E(\HH)$,
we let $\HH-X$ denote the spanning subhypergraph of $\HH$
with $E(\HH-X)=E(\HH)-X$.

Berge~\cite{Berge1970}
introduced the notion of a Berge-path and a Berge-cycle.
A Berge-path in a hypergraph $\HH$ is an alternate sequence
$P=(v_1, e_1, v_2, e_2, \ldots, v_l, e_l, v_{l+1})$
of distinct vertices $v_1, v_2,\dots, v_{l+1}$ and distinct edges
$e_1, e_2,\dots, e_l$
with $v_i, v_{i+1} \in e_i$
($1\le i \le l$).
We say that $P$ connects $v_1$ and $v_{l+1}$.
Similarly,
a Berge-cycle is
an alternate sequence $C=(v_1, e_1, v_2, e_2,\dots, v_l, e_l, v_{l+1}=v_1)$
of distinct vertices $v_1, v_2,\dots, v_l$ and distinct edges
$e_1, e_2,\dots, e_l$ with $v_i, v_{i+1} \in e_i$
($1\le i\le l$).
In both definitions,
the number $l$ of edges is called the length of $P$ and $C$, respectively.
A hypergraph $\HH$ is connected
if for every pair of vertices $u$ and $v$,
there exists a Berge-path connecting $u$ and $v$.
Moreover, for a positive integer $\lambda$,
we say that $\HH$ is {\em $\lambda$-edge-connected}
if for every set $\{e_1,\dots, e_l\}$ of at most $\lambda -1$ edges of  $\HH$,
$\HH-\{e_1,\dots, e_l\}$ is connected.

A Berge-path and a Berge-cycle are natural extensions of
a path and a cycle in ordinary graphs,
respectively,
and numerous topics on paths and cycles in graphs
have been extended to those for Berge-paths and Berge-cycles~in hypergraphs. 
See \cite{FKL-2020, GL-2012, KLLZ-2022},
for example.

In graph theory,
while a path and a cycle are defined as sequences,
we naturally view them as subgraphs.
In order to introduce the same view to hypergraphs,
Gerbner and Palmer~\cite{GP-2015} proposed 
 the notion of a Berge-$G$ hypergraph.
Let $\HH$ be a hypergraph and let $G$ be a graph
with $V(\HH)=V(G)$.
Then we say that $\HH$ is Berge-$G$
if there exists a bijection $f\colon E(G)\to E(\HH)$
satisfying $e\subseteq f(e)$ for every $e\in E(G)$.
Under these definitions,
a Berge-path (resp. Berge-cycle) in a hypergraph $\HH$ is a subhypergraph
of $\HH$ 
which is Berge-$P$ (resp. Berge-$C$)
for some path $P$ (resp. cycle $C$).

For a positive integer $k$,
a $k$-factor of a graph $G$
is a $k$-regular spanning subgraph of $G$.
As an extension of this definition,
we define a {\em Berge $k$-factor} of a hypergraph $\HH$
as a spanning subhypergraph $\FF$ of $\HH$
which is Berge-$G$ for some $k$-regular graph $G$.
Namely, if $\FF$ is a Berge $k$-factor of $\HH$, then there exists a mapping 
$\rho : E(\FF) \to \binom{V(\HH)}{2}$ such that 
(i) $\rho (e)=v_1v_2$, where $e\in E(\FF)$, $v_1,v_2\in e$ and $v_1v_2$ is an edge joining $v_1$ to $v_2$, and (ii)  a graph $F$ with $V(F)=V(\HH)$ and
$E(F)=\{\rho(e): e\in E(\FF)\}$
is a $k$-regular graph, and vice versa.

Using this definition,
we can extend various topics in factors of graphs
to Berge $k$-factors of hypergraphs.
Actually,
we are not the first to take this approach.
Gao,
Shan and Yu~\cite{GSY-2023}
extended the notion of toughness to hypergraphs
and proved that for every positive integer $k$,
a $k$-tough hypergraph $\HH$ with $k |V(\HH)|$ even and $|V(\HH)|\ge k+1$ has a Berge $k$-factor.

In this paper,
we consider Berge $k$-factors in regular hypergraphs with specified
edge-connectivity.
Let $r$ be a positive integer.
Petersen~\cite{Petersen}
proved that if $r$ is even,
then every $r$-regular graph
has a $k$-factor for every even integer $k$
with $2\le k\le r-2$.
In the case where $k$ is odd,
Gallai~\cite{Gallai}
proved that every $\lambda$-edge-connected
$r$-regular graph $G$ of even order
has a $k$-factor for every odd $k$
with $\frac{1}{\lambda}r\le k\le \left(1-\frac{1}{\lambda}\right)r$.
These results stimulated the research
on the value $k$ such that
for given $r$ and $\lambda$,
every $\lambda$-edge-connected $r$-regular graph $G$
with $k|V(G)|$ even has a $k$-factor.
After a number of studies,
this problem was completely solved.

\begin{theorem}[\cite{BSW-1985} (See also Theorem~3.4 of \cite{AK-2011})]
Let $r$ and $\lambda$ be positive integers
and let $k$ be an integer with $1\le k\le r-1$.
Also let $\lambda^{\ast}$ be the odd integer
with $\lambda^{\ast}\in\{\lambda, \lambda+1\}$.
Then every $\lambda$-edge-connected $r$-regular graph $G$
with $k|V(G)|$ even has a $k$-factor if 
\begin{enumerate}
\item
both $r$ and $k$ are even and $2\le k\le r-2$
\textnormal{(Petersen~\cite{Petersen})},
\item 
$r$ is even,
$k$ is odd and
$\frac{1}{\lambda}r\le k \le \left(1-\frac{1}{\lambda}\right)r$
\textnormal{(Gallai~\cite{Gallai})},
\item
$r$ is odd,
$k$ is even and
$2\le k\le \left(1-\frac{1}{\lambda^{\ast}}\right)r$, \quad or
\item
both $r$ and $k$ are odd and
$\frac{1}{\lambda^{\ast}}r\le k\le r-2$.
\end{enumerate}
Moreover, the above bounds on $k$ are best possible.
\label{regular_factors}
\end{theorem}

The statements of Theorem~\ref{regular_factors}
are divided into four cases depending on
the parities of $r$ and $k$.
But not all of them are independent.
If an $r$-regular graph $G$ has a $k$-factor $F$,
then $G-E(F)$ is an $(r-k)$-factor of $G$.
From this observation,
we see that the statement~(4) follows from the statement~(3).
The statement~(2) gives both nontrivial lower and upper bounds.
But the lower bound is deduced from the upper bound.

The purpose of this paper is to consider the same problem for hypergraphs,
and compare the obtained bounds of $k$ with the traditional bounds
for ordinary graphs.

The following is our result.

\begin{theorem}
Let $r$, $k$ and $\lambda$ be integers with
$2 \le \lambda \le r$ and $1\le k \le r-1$.
Also
let $\lambda^{\ast}$ be the odd integer with $\lambda^{\ast}\in \{\lambda, \lambda+1\}$.
Then every $\lambda$-edge-connected $r$-regular  hypergraph $\HH$
has a Berge $k$-factor if 
\begin{enumerate}
\item
$k$ is even and ~$\displaystyle 2\le k \le \min\left\{1-\frac{1}{\lambda^{\ast}}, \frac{2}{\rank(\HH)}\right\} \cdot r$,
\item
$k$ is odd,
$|V(\HH)|$ is even
and ~
$\displaystyle \frac{1}{\lambda} r\le k\le \min\left\{1-\frac{1}{\lambda}, \frac{2}{\rank(\HH)}\right\} \cdot r$, \quad or 
\item 
both $r$ and $k$ are even,  every edge in $\HH$ consists of
an even number of vertices, and ~$\displaystyle k \le  \frac{2r}{\rank(\HH)}$.
\end{enumerate}
\label{main_theorem}
\end{theorem}

While the upper bounds of $k$ involving $r$ and $\lambda$
in the conditions~(1) and~(2)
look similar to those in (2) and (3) of Theorem~\ref{regular_factors},
a new upper bound $\frac{2r}{\rank(\HH)}$ arises.
The ordinary graphs correspond to the case of $\rank(\HH)=2$
and this bound is trivial.
However,
for a hypergraph $\HH$ with $\rank(\HH) \ge 3$,
it affects the other upper bound.
We also note that there is no lower bound based on $\rank(\HH)$.
We discuss these phenomena in more detail in Section~4.

In the next section,
we introduce several tools we use in the proof of Theorem~\ref{main_theorem}.
We give a proof of Theorem~\ref{main_theorem} in Section~3.
In Section~4,
we make discussions on the obtained bounds.

For standard graph-theoretic terminology and notation
not explained in this paper,
we refer the reader to \cite{CLZ-2011}.
Just as we deal with hypergraphs with possible multiple edges,
when we say a graph,
we allow multiple edges.
For a graph $G$ and two disjoint  vertex sets $A$ and $B$ of $G$,
we denote by $E_G(A, B)$ the set of edges in $G$ joining $A$ to $B$,
and we define $e_G(A, B)$  by
$e_G(A, B)=|E_G(A, B)|$.
If $A$ is a singleton set consisting of $v$,
we write $E(v, B)$ and $e(v, B)$
instead of $E(\{v\}, B)$ and $e(\{v\}, B)$,
respectively.
For a vertex $v$ in a graph $G$,
we denote by $N_G(v)$ the neighborhood of $v$ in $G$.
If $S\subset V(G)$,
we often write $N_G(S)$ instead of $\bigcup_{v\in S} N_G(v)$.
If there is no fear of confusion,
we often identify a subgraph of $G$ with its vertex set.
For example,
if $A\subset V(G)$ and $H$ is a subgraph of $G-A$,
we often write $E_G(A, H)$ instead of $E_G(A, V(H))$.

%%%%%%%%%%%%%%%%%%%%%%%%%%%%%%%%%%%%%%%%%%%%%%
\section{Tools for Proof}
%%%%%%%%%%%%%%%%%%%%%%%%%%%%%%%%%%%%%%%%%%%%%%%%%%%%%

In this section,
we explain three tools we use in the proof of Theorem~\ref{main_theorem}~:~%
incidence graph,
parity $(g, f)$-factor
and discharging.

\subsection{Incidence Graph}

The {\em incidence graph} of a hypergraph $\HH$
is the bipartite simple graph $B=B(X, Y)$ with
$X=V(\HH)$ and $Y=E(\HH)$,
in which $v\in X$ and $e\in Y$ are adjacent if $v\in e$.
The incidence graph is a convenient tool
when we convert a problem on hypergraphs
to that on ordinary bipartite graphs.
See \cite{KLLZ-2022} and \cite{KLZ-2020},
for example.
Note that an alternate sequence of vertices and edges in $\HH$
is a Berge-path if and only if it is a path
joining two vertices in $X$ as a sequence of vertices in $B$.
Therefore,
$\HH$ is connected if and only if $B$ is connected.

Even if $\HH$ is $\lambda$-edge connected, $B$ is not always $\lambda$-edge-connected.
However, the edge-connectivity of $\HH$ gives some connectivity property
to its incidence graph.

\begin{lemma}
Let $\HH$ be a $\lambda$-edge-connected hypergraph
and let $B=B(X, Y)$ be the incidence graph of $\HH$ with $X=V(\HH)$ and $Y=E(\HH)$.
Let $\emptyset \ne R\subset V(B)$ and $D$ be a component of $B-R$.
If $X \setminus V(D) \ne \emptyset$ and $X\cap V(D) \ne \emptyset$, 
then $e_B(R, D)\ge \lambda$, in particular, if $R\cap X \ne \emptyset$ and  $D\ne \{y\}$ for any $y\in Y$, then $e_B(R, D)\ge \lambda$.
\label{Between_R_and_D}
\end{lemma}

\begin{proof}
Let $E_B(R, D)=\{e_1,\dots, e_l\}$.
Since $B$ is a bipartite graph,
we can write $e_i=x_iy_i$,
where $x_i \in X$ and $y_i\in Y$
($1\le i\le l$).
Then $\{y_1,\dots, y_l\} \subset E(\HH)$.
Note that $y_1,\dots, y_l$ are not necessarily distinct.

Assume that $X \setminus V(D) \ne \emptyset$ and $X\cap V(D) \ne \emptyset$.  Take two vertices $v_1\in X\cap V(D)$ and $v_2\in X\setminus V(D)$. Then, since every edge of $B$ connecting $V(D)$  and $V(B)-V(D)$ is contained in  $E_B(R, D)$,
$v_1$ and $v_2$ are not joined by a path in $B - E_B(R,D)$,
namely,  two vertices $v_1$ and $v_2$ of $\HH$ are not joined by a Berge-path in 
$\HH-\{y_1,\dots, y_l\}$. This implies  $l \ge \lambda$ because $\HH$ is $\lambda$-edge-connected.
Therefore, $e_B(R, D)\ge \lambda$.

Next assume that $R\cap X \ne \emptyset$ and $D\ne\{y\}$ for any $y\in Y$.
Then, $X\setminus V(D) \ne \emptyset$, and since $B$ is a bipartite graph, $V(D)\cap X\ne\emptyset$.
Hence $e_B(R, D)\ge \lambda$.
\end{proof}

\subsection{Parity $\boldsymbol{(g, f)}$-Factor}

Lov\'{a}sz~\cite{Lov-70} introduced a $(g, f)$-factor as an extension of Tutte's
$f$-factor~\cite{Tutte-52, Tutte-54} and gave a criterion for its existence.
He later studied its variant,
a parity $(g, f)$-factor.

Let $G$ be a graph and let $g$ and $f$ be two integer-valued
functions defined on $V(G)$
satisfying $g(v)\le f(v)$ and $g(v)\equiv f(v)\pmod{2}$ for every $v\in V(G)$.
Then a spanning subgraph $F$ of $G$ is called
a {\em parity $(g, f)$-factor} 
if $g(v)\le \deg_F(v)\le f(v)$ and $\deg_F(v)\equiv f(v)\pmod{2}$
hold for every $v\in V(F)=V(G)$.
Lov\'{a}sz~\cite{Lov-72} gave a necessary and sufficient condition
for a graph $G$ to have a parity $(g, f)$-factor.
For an integer-valued function $h$ on $V(G)$ and $X\subset V(G)$,
we define $h(X)$ by 
\[h(X)=\sum_{x\in X} h(x), \quad \mbox{and}\quad \deg_G(X)=\sum_{x\in X}\deg_G(x).\]

\begin{theorem}[Lov\'{a}sz~\cite{Lov-70}, Theorem~6.1 in~\cite{AK-2011}]
Let $G$ be a graph, and let $g$ and $f$ be two integer-valued functions
defined on $V(G)$
satisfying $g(v)\le f(v)$ and $g(v)\equiv f(v)\pmod{2}$ for all $v\in V(G)$.
Then $G$ has a parity $(g, f)$-factor
if and only if for every disjoint vertex sets $S$ and $T$ of $G$,
\[
\eta(S, T) := f(S)+\deg_G(T)-g(T)-e_G(S, T)-q(S, T)\ge 0,
\]
where $q(S, T)$ denotes the number of components $D$
of $G-(S\cup T)$ that satisfy
\begin{equation}
f(D)+e_G(D, T) \equiv 1\pmod{2}.
\label{f_odd_component}
\end{equation}
\label{parity_gf_factor}
\end{theorem}

\vspace{-2ex}
We call a component $D$ satisfying~(\ref{f_odd_component})
an {\em $f$-odd component} of $G-(S\cup T)$.
Note that Theorem~\ref{parity_gf_factor} holds
even if $g(v) < 0$ or $f(v) > \deg_G(v)$ for some $v\in V(G)$.

We relate a Berge $k$-factor of a hypergraph $\HH$
with a parity $(g, f)$-factor of its incidence graph.
Suppose a hypergraph $\HH$ and a positive integer $k$ are given.
Let $B=B(X, Y)$ be the incidence graph of $\HH$ with $X=V(\HH)$ and $Y=E(\HH)$.
Then we define two integer-valued functions $g$ and $f$ on $V(B)$ as
\[
g(v)= \bigg\{
\begin{array}{ll} 
 k & \mbox{if~ $v \in X$,} \\ 
-2N & \mbox{if~ $v \in Y$,}
\end{array}  
\quad \mbox{and} \quad  
f(v)=  \bigg\{
\begin{array}{ll} 
 k & \mbox{if~ $v \in X$,} \\ 
 2 & \mbox{if~ $v \in Y$,}
\end{array} 
\]
where $N$ is a sufficiently large integer.
We call $(g, f)$ {\em the pair of associated functions}.

Assume that $B$ has a parity $(g, f)$-factor $F$.
Then $\deg_F(x)=k$ for every $x\in X$ and
$\deg_F(y)\in \{0, 2\}$ for every $y\in Y$. 
Then the multiset of edges 
\[ \bigcup_{y\in Y, \; \deg_F(y)=2}  \{x_1x_2 : N_F(y)=\{x_1,x_2\} \} \]
forms a Berge $k$-factor of $\HH$. 
Conversely, assume that $\HH$
 has a Berge $k$-factor $\FF'$, and  $\FF'$ is Berge-$G$ for some $k$-regular graph $G$ with $V(G)=X$. Then, by letting a bijection $f\colon E(G)\to E(\FF')$, which satisfies $e\subseteq f(e)$ for every $e\in E(G)$, the edge set 
 \[
\bigcup_{e=x_1x_2\in E(G)}
\{x_1y, x_2y: ~\{x_1,x_2\}\subseteq y=f(e)\in Y\}
 \]
and the vertex set $X\cup Y$ forms a parity $(g,f)$-factor of $B$. Hence the following lemma holds.

\begin{lemma}
Let $\HH$ be a hypergraph and let
$B=B(X, Y)$ be the incidence graph of $\HH$,
where
$X = V(\HH)$ and $Y=E(\HH)$.
Let $(g, f)$ be the pair of associated functions.
Then $\HH$ has a Berge $k$-factor
if and only if $B$ has a parity $(g, f)$-factor.
\label{k_factor_parith_gf_factor}
\end{lemma}

\subsection{Discharging}

Egawa,
Kano and Ozeki~\cite{EKO-2025}
devised a method to interpret Tutte's $f$-Factor Theorem~\cite{Tutte-52, Tutte-54}
in the context of discharging.
We tailor their method to fit for a parity $(g, f)$-factor.

Let $G$ be a graph and let $g$ and $f$ be two integer-valued functions
defined on $V(G)$.
Also,
let $S$ and $T$ be a pair of disjoint vertex sets of $G$
and let $\CC$ be a set of components of $G-(S\cup T)$.
Then define $\eta_0(S, T, \CC)$
by
\[
\eta_0(S, T, \CC) = f(S)+\deg_G(T)-g(T)-e_G(S, T)-|\CC|.
\]

In this setting,
we define the initial charge $\varphi$ corresponding to $\eta_0(S, T, \CC)$
as a function $\varphi\colon S\cup T\cup \CC\to\mathbb{Q}$ satisfying
\begin{alignat*}{3}
\varphi(v) & = f(v) & \quad & \text{if $v\in S$}, \\
\varphi(v) & = \deg_G(v)-g(v)-e_G(v, S) & \quad & \text{if $v\in T$}, \\
\varphi(C) & = -1  & \quad &  \text{if $C\in \CC$}.
\end{alignat*}
Note $\eta(S, T)=\eta_0(S, T, \DD)$,
where $\DD$ is the set of $f$-odd components of $G-(S\cup T)$.

\begin{lemma}
$\displaystyle \sum_{x\in S\cup T\cup \CC} \varphi (x)=\eta_0(S, T, \CC)$
\label{total_charge}
\end{lemma}

\begin{proof} 
\begin{align*}
\sum_{x\in S\cup T\cup\CC} \varphi (x)
&= \sum_{v\in S}\varphi(v)
+\sum_{v\in T}\varphi(v)
+\sum_{C\in \CC} \varphi(C)\\
&= \sum_{v\in S} f(v) +\sum_{v\in T} \big(\deg_G(v)-g(v)-e_G(v, S) \big)-|\CC|\\
&= f(S)+\deg_G(T)-g(T)-e_G(T, S)-|\CC| \\
& =\eta_0(S, T, \CC).
\end{align*}
\end{proof}

In the proof of Theorem~\ref{main_theorem},
we assume that a hypergraph $\HH$ satisfying
the hypothesis does not have a Berge $k$-factor
and take the incidence graph $B$ of $\HH$.
Then $B$ does not have a parity $(g, f)$-factor,
where $(g, f)$ is the pair of associated functions,
which implies $\eta(S, T) < 0$ for some $S, T\subset V(B)$
with $S\cap T=\emptyset$.
In this setting,
we set the initial charge corresponding to $\eta(S, T)$ and
perform discharging.
A detail is given in the next section.

Note that we can associate the initial charge
to every integer-valued functions $g$ and $f$ on $V(G)$
and any set of components of $G-(S, T)$.
By imposing suitable conditions to $g$,
$f$ and $\CC$,
we can apply this method not only to 
parity $(g, f)$-factors but also to
$(g, f)$-factors and $f$-factors.

%%%%%%%%%%%%%%%%%%%%%%%%%%%%%%%%%%%%%%%%%%%%%%%%
\section{Proof of Theorem~\ref{main_theorem}}
%%%%%%%%%%%%%%%%%%%%%%%%%%%%%%%%%%%%%%%%%%%%%%%%%%

In this section, we prove Theorem~\ref{main_theorem} 
using three tools introduced in Section~2.

\begin{proof}[Proof of Theorem~\ref{main_theorem}]

Let $\HH$ be a $\lambda$-edge-connected $r$-regular 
hypergraph with $\rank(\HH)=t$.
Assume that $\HH$ satisfies one of the conditions~(1),
(2) and (3) of Theorem~\ref{main_theorem},
but does not have a Berge $k$-factor.
Let $B=B(X, Y)$ be the incidence graph of $\HH$
with $X=V(\HH)$ and $Y=E(\HH)$.

We note $\deg_B(x)=\deg_{\HH}(x)=r$ for every $x\in X$.
Also,
since $k\le\frac{2}{t}r$,
we have 
\begin{align}
 \deg_B(y)\le \rank(\HH)=t \le \frac{2r}{k} \quad \mbox{for every} \quad y\in Y. 
 \label{eq-rank-t}
 \end{align}

Let $(g, f)$ be the pair of associated functions.
Since $\HH$ does not have a Berge $k$-factor,
$B$ does not have a parity $(g, f)$-factor by Lemma~\ref{k_factor_parith_gf_factor}.
Then, by Theorem~\ref{parity_gf_factor},
we have
$\eta(S, T) < 0$
for some disjoint vertex sets $S$ and $T$ of $B$.

Let $\DD=\{D_1,\dots, D_m\}$ be the set of
$f$-odd components of $G-(S\cup T)$, where $m=q(S, T)$.

\begin{claim}
$S\cup T \ne \emptyset$ and $T\subseteq X$.
\label{T_subset_X}
\end{claim}

\begin{proof}
If $S\cup T=\emptyset$, then $f(B)=k|X|+2|Y| \equiv 0\pmod{2}$,  and so $q(\emptyset, \emptyset)=0$ by (\ref{f_odd_component}). Thus $\eta(\emptyset, \emptyset)=q(\emptyset, \emptyset)=0$,
which contaradicts $\eta(S, T) < 0$. Hence $S\cup T \ne \emptyset$.

Assume $T\cap Y\ne\emptyset$.
Then
\[
\begin{split}
\eta(S, T) &= f(S)+\deg_B(T)-g(T)-e_G(S, T)-q(S, T)\\
& = f(S)+\deg_B(T) -k|T\cap X|+2N|T\cap Y|-e_G(S, T)-q(S, T)\\
& > 0, 
\end{split}
\]
because $N$ is sufficiently large.
This is a contradiction.
\end{proof}

By Claim~\ref{T_subset_X},
we have
\begin{equation}
\eta(S, T) = k|S\cap X|+2|S\cap Y|+r|T|-k|T|-e_G(S, T)-m < 0.
\label{deficiency}
\end{equation}

\begin{claim}
$T\ne\emptyset$.
\label{T_nonempty}
\end{claim}

\begin{proof}
Suppose to the contary that $T=\emptyset$. Then $S\ne \emptyset$. 
If $D_i=\{e\}$ for some $e\in Y=E(\HH)$,
then $f(D_i)+e_B(D_i, \emptyset)=f(e)+0=2$,
which contradicts (\ref{f_odd_component}).
Therefore,
$D_i\ne\{e\}$ for any $e\in Y$.
Since $B$ is bipartite,
this implies $D_i\cap X\ne\emptyset$
for every $1\le i\le m$.

Assume that $k$ is even.
If $m \ge 1$,
then
\[
f(D_1)+e_B(D_1, \emptyset)
=k|D_1\cap X|+2|D_1\cap Y|\equiv 0\pmod{2},
\]
which contradicts (\ref{f_odd_component}). Thus $m=0$.
Then $\eta(S, \emptyset)=k|S\cap X|+2|S\cap Y|\ge 0$,
a contradiction.
Therefore, we may assume that $k$ is odd, in particular, 
the condition (2) of Theorem~\ref{main_theorem} holds.

Assume that $S\cap X=\emptyset$.
Then it follows from (\ref{deficiency}) that 
$\eta(S,\emptyset) = 2|S\cap Y| -m < 0$.
This implies that $m >2$ and $|S\cap Y| < \frac{m}{2}$. Since $D_i\cap X\ne \emptyset$ and $\emptyset \ne D_j\cap X \subseteq X\setminus D_i$ for $i\ne j$, $e_B(D_i,S)=e_B(D_i,S\cap Y) \ge \lambda$ for every $1\le i \le m$ by Lemma~\ref{Between_R_and_D}. 
So we have by (\ref{eq-rank-t}) that 
\[ m\lambda \le \sum_{1\le i \le m} e_B(D_i,S\cap Y) \le \deg_B(S\cap Y) \le t |S\cap Y| < \frac{tm}{2} .\]
Thus $\lambda < \frac{t}{2}$.  On the other hand, $\lambda \ge \frac{r}{k}$ by the condition (2) of Theorem~\ref{main_theorem}, and so $\frac{r}{k}< \frac{t}{2}$.
This implies $\frac{2r}{k} <t$, which contradicts (\ref{eq-rank-t}).

Therefore $S\cap X\ne\emptyset$.
Hence $e_B(S, D_i)\ge\lambda$ for every $1\le i\le m$ by $D_i\cap X\ne \emptyset$ and Lemma~\ref{Between_R_and_D}.
Since $\deg_B(x)=r$ for every $x\in X$
and $\deg_B(y) \le t \le\frac{2r}{k}$
for every $y\in Y$,
we have
$\sum_{x\in S\cap X}\deg_B(x) = r|S\cap X|$
and $\sum_{y\in S\cap Y}\deg_B(y)\le\frac{2r}{k}|S\cap Y|$.
By the condition (2) of Theorem~\ref{main_theorem}, $\frac{k}{r}\lambda\ge 1$.
Thus, we have
\[
\begin{split}
\eta(S,\emptyset)
&= k|S\cap X|+2|S\cap Y|-m \\
& \ge \frac{k}{r}\sum_{x\in S\cap X}\deg_B(x)
+\frac{k}{r}\sum_{y\in S\cap Y}\deg_B(y)-m\\
& =\frac{k}{r}\sum_{v\in S}\deg_B(v)-m \\
& \ge\frac{k}{r}\sum_{i=1}^m e_B(S, D_i)-m\\
& \ge \frac{k}{r}\cdot \lambda m-m \\
& \ge 0. 
\end{split}
\]
This is a contradiction,
and the claim follows.
\end{proof}

By Claims~\ref{T_subset_X} and~\ref{T_nonempty},
we have $(S\cup T) \cap X\ne\emptyset$, which implies that a condition in Lemma~\ref{Between_R_and_D} holds.

Let $\varphi\colon S\cup T\cup\DD\to\mathbb{Q}$
be the initial charge corresponding to $(g, f)$.
Since $T\subseteq X$ by Claim~\ref{T_subset_X},
we have
\begin{align*}
\varphi (v) &= 
\begin{cases}
k & \text{for $v\in S\cap X$}\\
2 & \text{for $v\in S\cap Y$}\\
r-k-e_B(v, S) & \text{for $v\in T$}
\end{cases}\\
\intertext{and}
\varphi(D_i) &= -1\quad(1\le i\le m).
\end{align*}
We set the following rules
and perform discharging.
\begin{enumerate}
\item[(i)]
For each edge $xz$ joining $x\in S\cap X$ to $z\in V(D_i)$,
$x$ sends a charge of $\dfrac{k}{r}$ to $D_i$.
\item[(ii)]
For each edge $yz$ joining $y\in S\cap Y$ to $z\in V(D_i)$,
$y$ sends a charge of $\dfrac{2}{\deg_B(y)}$ to $D_i$.
\item[(iii)]
For each edge $xz$ joining $x\in T$ to $z\in V(D_i)$,
$x$ sends a charge of $\dfrac{r-k}{r}$ to $D_i$.
\item[(iv)]
For each edge $yx$ joining $y\in S\cap Y$ to $x\in T$,
$y$ sends a charge of $\dfrac{2}{\deg_B(y)}$ to $x$.
\end{enumerate}

Let $\varphi^*\colon S \cup T\cup \DD\to \mathbb{Q}$ be the charge
resulting from the discharge.

\begin{claim}
$\varphi^*(D_i)\ge 0$ for each $1\le i\le m$.
\label{D_part}
\end{claim}

\begin{proof}
Since $\deg_B(y)\le t \le\frac{2r}{k}$ for $y\in S\cap Y$ by (\ref{eq-rank-t}),
the discharging rules~(i), (ii) and (iii) give
\begin{align}
\varphi^* (D_i) &=
-1 + \frac{k}{r}e_B(S\cap X, D_i)
+\sum_{yv\in E_B(S\cap Y, D_i)} \frac{2}{\deg_B(y)} ~ +\frac{r-k}{r}e_B(T, D_i) \nonumber \\
&\ge -1+\frac{k}{r}e_B(S\cap X, D_i)+\frac{k}{r}e_B(S\cap Y, D_i)
+\frac{r-k}{r}e_B(T, D_i) \nonumber \\
&= -1+\frac{k}{r}e_B(S, D_i)+\frac{r-k}{r}e_B(T, D_i). \label{eq-phiD}
\end{align}
Hence, if $e_B(S, D_i)\ge 1$ and $e_B(T, D_i)\ge 1$, then
$\varphi^*(D_i)\ge 0$.
Therefore, we may assume 
\begin{align}
e_B(S, D_i) = 0 \quad \mbox{or} \quad e_B(T, D_i)=0.
\label{eq-5}
\end{align}

Suppose that the condition~(3) of Theorem~\ref{main_theorem} holds.
Then every vertex in $B$ has an even degree,
and so
\[
0\equiv \sum_{v\in V(D_i)}\deg_B(v)=e_B(D_i, S\cup T)+2|E(D_i)|
\equiv e_B(D_i, S\cup T)\pmod{2}.
\]
On the other hand,
since $k$ is even and $D_i$ is an $f$-odd component,
\[
1 \equiv f(D_i)+e_G(D_i, T) = k|D_i\cap X|+2|D_i\cap Y|+e_B(D_i, T)
\equiv e_B(D_i, T)\pmod{2}.
\]
These imply $e_B(D_i, S)\equiv e_B(D_i, T)\equiv 1\pmod{2}$,
and hence $e_B(D_i, S)\ge 1$ and $e_B(D_i, T)\ge 1$,
which contradicts (\ref{eq-5}).
Therefore,
we may assume that the condition~(1) or~(2) of Theorem~\ref{main_theorem} holds.

First consider the case where $e_B(T, D_i)=0$ in (\ref{eq-5}).
Then
\[
1 \equiv f(D_i)+e_B(T, D_i)
=k|D_i\cap X|+2|D_i\cap Y| \equiv k|D_i\cap X|\pmod{2}.
\]
This implies that $k$ is odd and $D_i\cap X\ne\emptyset$.
Since $\emptyset \ne T\subset X\setminus D_i$ and $e_B(T, D_i)=0$,
we have $e_B(S, D_i)\ge \lambda$
by Lemma~\ref{Between_R_and_D}.
Also,
since $k$ is odd,
the condition~(2) holds,
which implies $\frac{k}{r}\lambda \ge 1$.
Therefore, by (\ref{eq-phiD}), we have
\[
\varphi^*(D_i) \ge -1+\frac{k}{r}e_B(S, D_i)\ge -1+\frac{k}{r}\lambda \ge 0.
\]

Next consider the case where $e_B(S, D_i)=0$ in (\ref{eq-5}).
Assume first that $D_i\cap X\ne\emptyset$.
Then, since $\emptyset \ne T\subset X\setminus D_i$ by Claims~\ref{T_subset_X} and \ref{T_nonempty}, we have $e_B(T, D_i)\ge \lambda$
by Lemma~\ref{Between_R_and_D}.
Moreover,
since 
\[
1  \equiv f(D_i)+e_B(D_i, T) \equiv k|D_i\cap X|+e_B(D_i, T)\pmod{2},
\]
we have $e_B(T, D_i)\ge \lambda^*$ 
if $k$ is even.
Then
since $k\le \left(1-\frac{1}{\lambda^*}\right)r$
if $k$ is even and
$k\le\left(1-\frac{1}{\lambda}\right)r$ if $k$ is odd,
we have by (\ref{eq-phiD}) that

\[
\begin{split}
\varphi^*(D_i) & \ge -1+\frac{r-k}{r}e_B(T, D_i)\\
&\ge
\begin{cases}
-1+\frac{r-k}{r}\lambda^* ~~ \ge 0 &\quad  \text{if $k$ is even}\\
-1+\frac{r-k}{r}\lambda ~~~ \ge 0 & \quad \text{if $k$ is odd.}
\end{cases}
\end{split}
\]

Next assume that  $D_i\cap X=\emptyset$.
Then $D_i = \{e\}$ for some $e\in Y$.
By the assumption that $|e|\ge 2$, we  have
 $e_B(\{e\}, T) =\deg_B(e)=|e| \ge 2$.
On the other hand,
since $D_i=\{e\}$ is an $f$-odd component,
we have 
\[
1\equiv f(e)+e_B(\{e\}, T)
=2+e_B(\{e\}, T)\equiv e_B(\{e\}, T)\pmod{2},
\]
and thus $e_B(\{e\}, T)\ge 3$.
This yields $|e|\ge 3$ and hence $t=\rank(\HH)\ge 3$.
Since each of the conditions~(1) and~(2) in Theorem~\ref{main_theorem} gives $k \le\frac{2}{t}r$,
we have $\frac{k}{r}\le\frac{2}{t}\le \frac{2}{3}$.
Therefore, it follows from (\ref{eq-phiD}) that
\[
\varphi^*(D_i)\ge -1+\frac{r-k}{r} e_B(\{e\}, T)
\ge -1+\left(1-\frac{2}{3}\right)\cdot 3=0.
\]
Therefore the claim is proved. 
\end{proof}

Let $U= V(D_1) \cup V(D_2) \cup \ldots \cup V(D_m)$.

\begin{claim}
For every $v\in S$,
$\varphi^*(v)\ge 0$.
\label{S_part}
\end{claim}

\begin{proof}
For $x\in S\cap X$,
by the discharging rule~(i),
we have
\[
\varphi^*(x)=k-\frac{k}{r}e_B(x, U)\ge k-\frac{k}{r}\deg_B(x)=k-\frac{k}{r}\cdot r=0.
\]

For $y\in S\cap Y$,
by the discharging rules~(ii) and~(iv),
we have
\[
\varphi^*(y)= 2-\frac{2}{\deg_B(y)}(e_B(y, T)+e_B(y, U))
\ge 2-\frac{2}{\deg_B(y)}\deg_B(y)=0.
\]
\end{proof}

\begin{claim}
For every $x\in T$,
$\varphi^*(x)\ge 0$.
\label{T_part}
\end{claim}

\begin{proof}
Note that since $T\subset X$,
we have $N_B(T)\cap S\subset Y$.
Also,
$\deg_B(y)\le t \le \frac{2r}{k}$ for every $y\in Y$.
Hence, for $x\in T$, we have
\begin{align*}
\varphi^*(x)
&= r-k-e_B(x, S)-\frac{r-k}{r}e_B(x, U)+\sum_{xy\in E_B(x, S)}\frac{2}{\deg_B(y)}\\
& \ge r-k-e_B(x, S)-\frac{r-k}{r}e_B(x, U)+\frac{k}{r}e_B(x, S)\\
&= r-k-\frac{r-k}{r}\bigl(e_B(x, S)+e_B(x, U)\bigr)\\
&\ge r-k-\frac{r-k}{r}\deg_B(x) \\
& =r-k-\frac{r-k}{r}\cdot r=0.
\end{align*}
\end{proof}

By Claims~\ref{D_part},
\ref{S_part} and \ref{T_part},
we have $\varphi^*(z)\ge 0$ for every $z\in S\cap T\cup \DD$,
and hence $\sum_{z\in S\cup T\cup\DD}\varphi^*(z)\ge 0$.
This contradicts $\sum_{z\in S\cup T\cup\DD}\varphi^*(z)
=\sum_{z\in S\cup T\cup \DD}\varphi(z) < 0$.
Consequently, the proof of Theorem~\ref{main_theorem} is complete.

\end{proof}

\section{Discussion}

In this section,
we discuss conditions~(1),
(2) and~(3) of Theorem~\ref{main_theorem} and compare them with the statements
of Theorem~\ref{regular_factors}.

The conditions~(1) and~(2) give two upper bounds,
the one based on the edge-connectivity
and the other based on the rank.
Though the latter bound looks technical in the proof,
it is actually not as we show in the following theorem.

\begin{theorem}
For every integer $k, r, t$
with $r\ge 2$,
$t\ge 2$ and $1\le k\le r$,
there exist infinitely many  $r$-edge-connected
$r$-regular $t$-uniform hypergraphs
which do not have a Berge $k$-factor for
any $k$ with $k > \frac{2}{t}r$.
\label{new_bound_essential}
\end{theorem}

\begin{proof}
Let $k$ be an integer with $k > \frac{2}{t}r$.
If $t = 2$,
then we have $k > r$,
and no $r$-regular graph has a $k$-factor.
Therefore,
we may assume $t\ge 3$.

For an integer $m$ with $m\ge 2$,
we construct a $t$-uniform hypergraph $\HH$ of order $tm$ in the following way.

Let $V(\HH) := \integerset/ tm\integerset=\{0, 1, \dots, tm-1\}$.
For suffices $i$ and $j$ with $0\le i\le r-1$ and $j\in \integerset/m\integerset$,
we define
$e_{i, j}=\{i+jt+l\colon 0\le l \le t-1\}$.
Let $E_i=\{e_{i, j} \colon 0\le j\le m-1\}$ (see Fig.\ref{number_line}).
Note that each $e_{i, j}$ contains $t$ vertices
and that each $E_i$ is a partition of $V(\HH)$.
Note also that $E_i=E_j$ if $i\equiv j\pmod{t}$.
Let $E(\HH)$ be the multiset defined by $E(\HH)=\bigcup_{i=0}^{r-1} E_i$.
By the construction,
$\HH$ is $t$-uniform.
Moreover,
since each $E_i$ is a partition of $V(\HH)$,
every vertex in $\HH$ is contained in exactly one edge in $E_i$.
This yields $\deg_{\HH}(v)=r$ and hence $\HH$ is $r$-regular.

\begin{figure}[t]
\centering
\includegraphics[width=420.6pt]{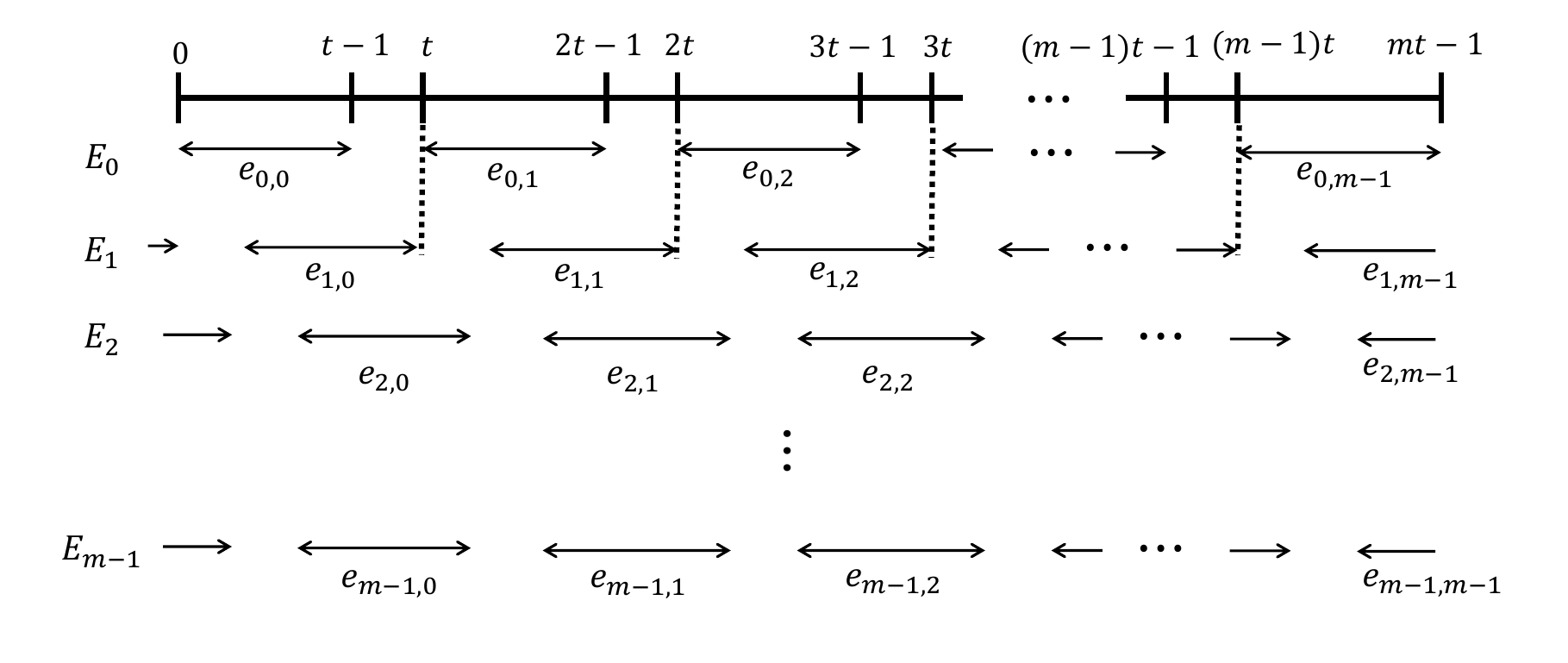}
%\vspace{-3zw}
\caption{The segments (two-way arrows) in the figure represent the edges of $\HH$.
The segments of $E_i$ are obtained from the segments of $E_{i-1}$ by cyclically shifting them one unit to the right.
}
\label{number_line}
\end{figure}

We prove that $\HH$ is $r$-edge-connected.
Let $E'\subset E(\HH)$ with $|E'|\le r-1$ and assume $\HH-E'$ is not connected.
Take $v, w\in V(\HH)$ such that
there does not exist a Berge-path
connecting $v$ and $w$ in $\HH-E'$.
Since $|E'|\le r-1$,
we can take $E_i$ with $E_i\cap E'=\emptyset$.
By the cyclic symmetry of the suffices of $e_{i, j}$,
we may assume $E_0\cap E'=\emptyset$.
Also by the symmetry of $v$ and $w$,
we may assume $v\in e_{0, j}$ and $w\in e_{0, j'}$ for some $j$ and $j'$
with $0\le j\le j'\le m-1$.
However,
if $j=j'$,
then $v, e_{0, j}, w$ is a Berge-path of length~$1$
in $\HH-E'$,
a contradiction.
Therefore,
we have $j < j'$.

For $0\le i\le m-1$,
define the multiset $F_i$ By
$F_i=\{e\in E(\HH)\colon \{it-1, it\}\subset e\}$.
By the definition of $E_i$,
$E_p\cap F_i\ne\emptyset$ if and only if $p\not\equiv 0\pmod{t}$.

Choose $f_i\in F_i$ for each $i$ with $0\le i\le m-1$ and
define two Berge-path $P$ and $Q$ by
\begin{align*}
P &= v, e_{0, j}, (j+1)t-1, f_{j+1}, (j+1)t, e_{0, j+1}, (j+2)t-1, f_{j+2}, (j+2)t,\\
& \dots, (j'-1)t-1, f_{j'-1}, (j'-1)t, e_{0,j'-1}, j't-1, f_{j'}, j't, e_{0, j'}, w\text{ and}\\
Q &= v, e_{0, j}, jt, f_j, jt-1, e_{0, j-1}, (j-1)t, f_{j-1}, (j-1)t-1, e_{0, j-2}, (j-2)t,\\
& \dots, (j'+2)t-1, e_{0, j'+1}, (j'+1)t, f_{j'+1}, (j'+1)t-1, e_{0, j'}, w.
\end{align*}
Since both $P$ and $Q$ connect $v$ and $w$,
we have $E'\cap E(P)\ne\emptyset$
and $E'\cap E(Q)\ne\emptyset$
for any choice of $f_i\in F_i$.
Therefore,
$F_{i_1}\subset E'$ for some $i_1$ with $i_1\in\{j+1, j+2,\dots, j'\}$
and $F_{i_2}\subset E'$ for some $i_2$ with $i_2\in\{j'+1, j'+2,\dots, j\}$.
Let $q$ and $\alpha$ be integers with $q\ge 0$,
$0\le\alpha\le t-1$ and $r=qt+\alpha$.
Since $F_p\cap E_i\ne\emptyset$ if and only if $p\not\equiv 0\pmod{t}$,
we have $|F_{i_1}|=|F_{i_2}|=q(t-1)+\max\{\alpha-1, 0\}$
and hence $|E'|\ge 2\bigl(q(t-1)+\max\{\alpha-1, 0\}\bigr)$.
On the other hand,
since $|E'|\le r-1=qt+\alpha-1$,
we have $qt+\alpha-1\ge 2\bigl(q(t-1)+\max\{\alpha-1, 0\}\bigr)$.

First,
suppose $\alpha\ge 1$.
Then we have $qt+\alpha-1\ge 2\bigl(q(t-1)+\alpha-1\bigr)$,
which is equivalent to $q(t-2)+\alpha\le 1$.
Since $\alpha\ge 1$ and $t\ge 3$,
this implies $\alpha = 1$ and $q = 0$.
However,
this yields $r=qt+\alpha=1$,
contradicting the hypothesis of $r\ge 2$.

Next,
suppose $\alpha = 0$.
Then we have $qt-1 \ge 2\bigl(q(t-1)\bigr)=2qt-2q$,
which yields $q(t-2)+ 1\le 0$.
However,
this is again a contradiction since $q\ge 0$ and $t\ge 3$.
Thus,
$\HH$ is $r$-edge-connected.

Suppose $\HH$ has a Berge $k$-factor.
Let $B=B(X, Y)$ be the incidence graph of $\HH$
with $X=V(\HH)$ and $Y=E(\HH)$.
By Lemma~\ref{k_factor_parith_gf_factor},
$B$ contains a spanning subgraph $F$
satisfying $\deg_F(x)=k$ for every $x\in X$
and $\deg_F(y)\in\{0, 2\}$
for every $y\in Y$.
This implies $k|X| =|E(F)| \le 2|Y|$.
On the other hand,
we have $|X|=|V(\HH)|=tm$
and $|Y|=E(\HH)=rm$.
Therefore,
we have $ktm \le 2mr$,
which implies $k\le r\cdot \frac{2}{t}$.
This contradicts the assumption of the theorem.
\end{proof}

Let $\rank(\HH)=t$. We compare the two upper bounds
$\left(1-\frac{1}{\lambda^*}\right)r$ and $\frac{2}{t}r$ in the condition (1) 
and $\left(1-\frac{1}{\lambda}\right)r$ and $\frac{2}{t}r$ in the condition (2) of Theorem~\ref{main_theorem}.
Note that since $\lambda\ge 2$,
we have $\lambda^*\ge 3$.
On the other hand,
if $t\ge 3$,
then $\frac{2}{t}\le \frac{2}{3} \le 1-\frac{1}{\lambda^*}$.
Thus,
if $k$ is even,
the upper bound $\frac{2}{t}r$ always beat $\left(1-\frac{1}{\lambda^*}\right)r$.
Therefore,
$\left(1-\frac{1}{\lambda^*}\right)r$
is a meaningful bound only if $t=2$,
the case of ordinary graphs.

For odd $k$,
the edge-connectivity matters only if $1-\frac{1}{\lambda} < \frac{2}{t}\le \frac{2}{3}$.
This implies $\lambda = 2$ and $t \le 3$.
Therefore,
$\left(1-\frac{1}{\lambda}\right)r$
is a meaningful bound only for
ordinary graphs and hypergraphs with edge-connectivity~$2$
and rank $3$.
It demonstrates a sharp contrast between ordinary graphs
and hypergraphs of rank $3$ or more.

Next,
we discuss the lower bound in the condition~(2)
of Theorem~\ref{main_theorem}.
Unlike the upper bound,
$\frac{1}{\lambda} r$ remains as a meaningful bound for hypergraphs of any rank.

\begin{theorem}
For every integers $r$,
$\lambda$  and $t$
with $r\ge 2$,
$2\le \lambda \le \frac{1}{2}r$
and $t\ge 2$,
there exist infinitely many $\lambda$-edge-connected $r$-regular hypergraphs $\HH$ of even order
with rank at least $t$
such that $\HH$ does not have a Berge $k$-factor for any odd integer $k$ with $k<\frac{1}{\lambda}r$.
\end{theorem}

\begin{proof}
Let $t' = 2\max\left\{\left\lceil\frac{r}{k}\right\rceil,  t\right\}$.
Then $t'$ is an even integer.
We take an arbitrary even integer $s$ with $s > t'$.
Prepare $t'+1$ pairwise disjoint sets $X_0, X_1,\dots, X_{t'}$
with $|X_0|=t'$ and $|X_1|=|X_2|=\dots =|X_t|=s$.
Let $X_0 = \{x_1,\dots, x_{t'}\}$
and $X= \bigcup_{i=0}^{t'} X_i$.
Let $Y_0$ be a multiset consisting of $\lambda$ copies of $X_0$,
and let $Y_i$ and $Y'_i$ be multisets consisting of $r-\lambda$ copies of $X_i\cup\{x_i\}$ and
$\lambda$ copies of $X_i$,
respectively
($1\le i\le t'$).
Finally,
let $Y=Y_0\cup\left(\bigcup_{i=1}^{t'}(Y_i\cup Y'_i)\right)$,
and define a hypergraph $\HH$ by $\HH=(X, Y)$.
Let $B=B(X, Y)$ be the incidence graph of $\HH$ (see Fig.~\ref{lower_bound}).
Then $|\HH| = |X| =st'+t'=t'(s+1)\equiv 0\pmod{2}$.
Also $x_i\in X_0$ is incident with the edges in $Y_0\cup Y_i$ and
$x\in X_i$ is incident with the edges in $Y_i\cup Y'_i$.
Therefore,
$\HH$ is an $r$-regular hypergraph of rank $s+1$.

We claim that $\HH$ is $\lambda$-edge-connected.
Let $F$ be a set of edges in $\HH$ with $|F|\le \lambda-1$
and let $\HH'=\HH-F$.
Since $\lambda \le \frac{1}{2}r$,
$r-\lambda \ge \lambda$ and hence
$Y_i-F\ne\emptyset$
($0\le i\le t'$).
Take $y_i\in Y_i-F$.
Then $x_i y_0 x_j$ is a Berge-path in $\HH'$ connecting
$x_i, x_j\in X_0$
for $(1\le i < j \le t'$)
and $x y_i x_i$ is a Berge-path in $\HH'$
connecting a vertex $x$ in $X_i$ and $x_i\in X_0$.
Therefore,
$\HH'$ is connected,
and hence $\HH$ is $\lambda$-edge-connected.

\begin{figure}[t]
\centering
\includegraphics[width=200.6pt]{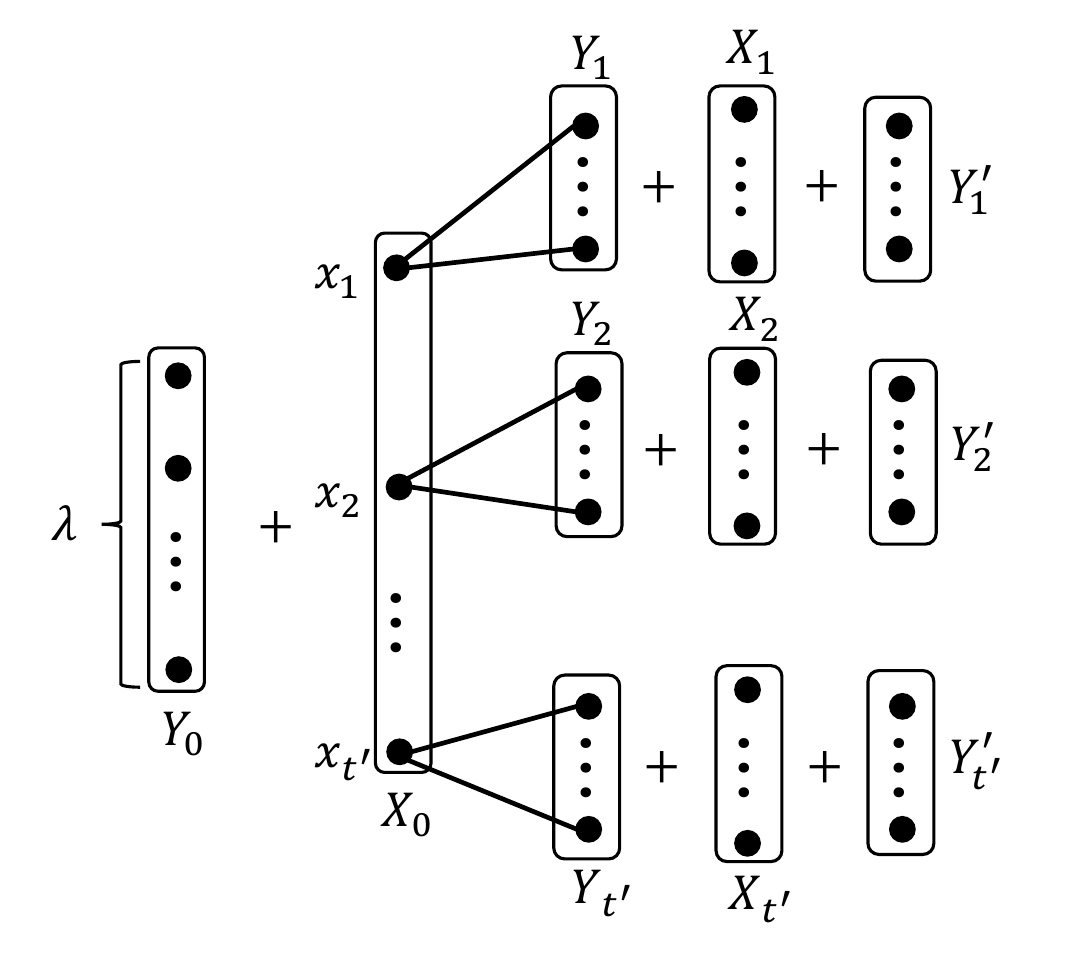}
%\vspace{-3zw}
\caption{The incidence graph of $\HH$.}
\label{lower_bound}
\end{figure}

Suppose $\HH$ has a $k$-factor for some odd $k$ with
$k < \frac{1}{\lambda}r$.
Then $B$ has a parity $(g, f)$-factor
with
$f(v)=k$ for every $v\in X$ and
$f(v)\in\{0, 2\}$ for every $v\in Y$.

We consider $\eta(Y_0,\emptyset)$.
The components of $B-Y_0$ is
the graphs $D_i$ induced by $\{x_i\}\cup Y_i\cup X_i\cup Y'_i$
($1\le i\le t'$).
Since $k$ is odd and $s$ is even,
\[
f(D_i)+e_B(D_i,\emptyset)
= k+ 2|Y_i|+k|X_i|+2|Y'_1|
\equiv k(s+1)\equiv 1\pmod{2}.
\]
Therefore,
$D_i$ is an $f$-odd component
and hence $q(Y_0, \emptyset)=t'$.
Now since $t'\ge 2\cdot \frac{r}{k}$
and $k < \frac{1}{\lambda}r$,
we have
\[
\eta(Y_0, \emptyset)
=2|Y_0|-t'=2\lambda -t' < 2\cdot \frac{r}{k}-2\cdot\frac{r}{k}=0.
\]
Therefore,
$\HH$ does not have a $k$-factor.
\end{proof}

Finally,
we discuss the condition~(3).
Among four statements in Theorem~\ref{regular_factors},
(1) is the only one that does not involve edge-connectivity.
We suspect  that a mechanism is different  from that of the other statements.
The condition~(3) seems to suggest that the parity of the number of vertices
contained in each edge plays an important role.

\bigskip \noindent
{\bf Acknowledgments}
The second author was supported by JSPS KAKENHI Grant Number JP22K13956 and JSPS KAKENHI Grant Number JP25K17301.
The third author was supported by JSPS KAKENHI Grant Number 24K06835. The fourth author was  supported by JSPS KAKENHI Grant Numbers  24K06836.  This work was supported by the Research Institute for Mathematical Sciences, an International Joint Usage/Research Center located in Kyoto University.

\end{document}